\documentclass[12pt]{amsart}

\usepackage{amsmath}
\usepackage{amsfonts}
\usepackage{amssymb,enumerate}
\usepackage{amsthm}
\usepackage{fancyhdr}
\usepackage{tikz}
\usetikzlibrary{arrows,matrix}
\usepackage{hyperref}
\usepackage{algorithm}
\usepackage{caption}
\usepackage[noend]{algpseudocode}
\usepackage{epstopdf}

\newtheorem{lemma}{Lemma}

\newtheorem{proposition}[lemma]{Proposition}
\newtheorem{theorem}[lemma]{Theorem}
\newtheorem{question}[lemma]{Question}

\makeatletter
\def\BState{\State\hskip-\ALG@thistlm}
\makeatother

\title[]{On integers which are representable as sums of large squares}

\author{Alessio Moscariello}

\subjclass[2010]{11D07,11P05.}

\keywords{Frobenius number, Four Square Theorem.}

\address[Alessio Moscariello]{Dipartimento di Matematica e Informatica, \ Universit\`a di Catania, \  Viale Andrea Doria 6, 
95125 Catania, Italy.}
\address[Alessio Moscariello]{Scuola Superiore di Catania, \ Universit\`a di Catania, \  Via Valdisavoia 9, 
95125 Catania, Italy.}

\email{alessio.moscariello@studium.unict.it}

\begin{document}
\maketitle
\begin{abstract}
We prove that the greatest positive integer that is not expressible as a linear combination with integer coefficients of elements of the set $\{n^2,(n+1)^2,\ldots \}$ is asymptotically $O(n^2)$, verifying thus a conjecture of Dutch and Rickett. Furthermore we ask a question on the representation of integers as sum of four large squares.

\end{abstract}
\section*{Introduction}
Additive number theory is the branch of arithmetic that studies particular subsets of integers (such as arithmetic progressions, prime numbers or squares) and their behavior under addition. A typical class of problems in this area concerns determining which positive integers can be expressed as sums of elements of a fixed subset of integers. In this aspect, one of the earliest results of additive number theory is Lagrange's Four Square Theorem, stating that every non-negative integer can be expressed as sum of four squares. Another classical problem is the Diophantine Frobenius Problem (cf. \cite{RA}), that asks for the greatest integer that is not representable as a linear combination with non-negative integer coefficients of elements of a subset of integers. In connection to both topics, it is interesting to understand which numbers are representable as sum of large squares.  In this work, we investigate the largest number that cannot be represented as a sum of squares of integers greater than a fixed $n$. In particular, we answer affirmatively a conjecture of Dutch and Rickett (cf. \cite{DR}), stating that the greatest integer that is not expressible as a sum of squares of integers greater than or equal to $n$ grows like $n^2$, namely we prove the following result:

\begin{theorem}\label{DR}
Let $\Gamma_n$ be the set of integers that can be expressed as a finite sum of squares of integers greater than or equal to $n$, and denote by $F(\Gamma_n)$ the greatest integer not belonging to $\Gamma_n$. Then $F(\Gamma_n)=O(n^2)$.
\end{theorem} 
Finally we ask a question related to the representation of an integer as a sum of four large squares. 
\section{Main result}
Denote by $\mathbb{N}$ the set of non-negative integers. Since $n^2$ and $(n+1)^2$ are relatively prime for any $n \in \mathbb{N}$, $n \neq 0$, then Sylvester's formula (cf. \cite{S}) states that every integer greater than or equal to $(n^2-1)[(n^2+1)-1]$ can be written as a non-negative linear combination of $n^2$ and $(n+1)^2$, thus the set $\mathbb{N}\setminus \Gamma_n$ is finite. Recently it has been proved (cf. \cite{LRS}) that the greatest integer that is not representable as a non-negative linear combination of $n^2,(n+1)^2,(n+2)^2$ is of order $O(n^3)$. Concerning our problem, it is known that $F(\Gamma_n)=o(n^{2+\epsilon})$ for any $\epsilon > 0$ (cf. \cite{DR}).  We estimate $F(\Gamma_n)$ by focusing on representations that involve the least possible number of squares. 
The well-known Four Square Theorem, due to Lagrange, states that every positive integer can be written as a sum of four squares.
Furthermore, if we denote by $$r(n) := \#\big\{(a_1,a_2,a_3,a_4) \in \mathbb{Z}^4 \ | \ a_1^2+a_2^2+a_3^2+a_4^2=n \big\}$$ and by $\sigma'(n)$ the sum of divisors of $n$ that are not divisible by $4$, in 1828 Jacobi proved that $r(n)$ is equal to $8\sigma'(n)$ (cf. \cite{J}). Let $\sigma(n)$ stand for the sum of divisors of $n$. Notice that if $n \in \mathbb{N}$ is such that $n \not \equiv 0 \pmod{4}$ then $\sigma'(n)=\sigma(n) > n$, hence for $n \rightarrow +\infty$, $n \not \equiv 0 \pmod{4}$ we get that $r(n) \rightarrow +\infty$. Therefore it makes sense to ask for representations as sum of four relatively large squares. However, the elements of the infinite set $\mathcal{B}=\{1,3,5,9,11,17,29,41\} \cup \{4^\alpha 2,4^\alpha 6, 4^\alpha 14 \ | \ \alpha \in \mathbb{N}\}$ cannot be expressed as sum of four non-zero squares  (see \cite{G}, Chapter 6, Theorem 3). Taking that into account, we pose the following question:

\begin{question}\label{fourlarge}
Is there $K$ in $\mathbb{N}$ such that for every $n$ in $\mathbb{N}$ there exists $a_1, \ldots, a_4$ such that $n = a_1^2 + a_2^2+a_3^2+a_4^2$ and for each $i$ either $a_i=0$ or $a_i \ge \frac{\sqrt{n}}{K}$?
\end{question}
Note that if such $K$ exists then, for every $n \in \mathbb{N}$, every positive integer $M \ge K^2n^2$ is a sum of four squares contained in the set $\{0,n^2,(n+1)^2,\ldots\}$: hence $M \in \Gamma_n$,  $F(\Gamma_n) \le K^2n^2$ and $F(\Gamma_n)=O(n^2)$. It is easy to see that such a $K$ cannot be less than $8$ (by taking $N=55$), and computational evidence suggests that $K=8$ affirmatively answers Question \ref{fourlarge}. We are going to give a partial answer to this question by proving that such a $K$ must exist, in order to prove our main theorem.
Consider now the $4$-dimensional sphere of radius $\sqrt{n}$ 
$$\mathcal{S}_n:=\{(x_1,x_2,x_3,x_4) \in \mathbb{R}^4 \ | \ x_1^2+x_2^2+x_3^2+x_4^2=n\}.$$
The representations of $n$ as sum of four integer squares correspond to the integral points of $\mathcal{S}_n$. Question \ref{fourlarge} is thus related to the distribution of these integral points over the sphere $\mathcal{S}_n$. This problem has been widely studied in literature in the more general context of quadratic forms: it has been proved (cf. \cite{I}, Theorem 11.5) that for an arbitrary quadratic form $Q$ if $r \ge 4$ then the vectors $a \in \mathbb{Z}^r$ with $Q(a)=n$ are asymptotically equidistributed as $n \rightarrow \infty$ over integers satisfying the congruence $Q(a) \equiv n \pmod{128(det(Q))^3}$. We will now explain the meaning of the words ``asymptotically equidistributed". First, points on $\mathcal{S}_n$ scale concentrically to points of the unit sphere $\mathcal{S}_1$. Thus we can study the distribution of integral points over $\mathcal{S}_n$ by considering the distribution of the associated scaled points of $\mathcal{S}_1$. For this purpose we define, for a domain $F \subseteq \mathcal{S}_1$, $$r_F(n) := \#\left\{(a_1,a_2,a_3,a_4) \in \mathbb{Z}^4 \ \big| \ (a_1,a_2,a_3,a_4) \in \mathcal{S}_n, \left( \frac{a_1}{\sqrt{n}},\frac{a_2}{\sqrt{n}},\frac{a_3}{\sqrt{n}},\frac{a_4}{\sqrt{n}}\right)  \in F\right\},$$ that is the number of integral points of $\mathcal{S}_n$ that scale to points of $F$. With this notation,  by ``asymptotically equidistributed" we mean that if $\mathcal{F}$ is a convex domain with smooth boundaries on $\mathcal{S}_1$, then $r_\mathcal{F}(n) \sim |\mathcal{F}|r(n)$, where $|\mathcal{F}|$ denotes the normalized surface area of $\mathcal{F}$ (assuming $|\mathcal{S}_1|=1$). Notice that, if we consider the quadratic form $Q(x)=x_1^2+x_2^2+x_3^2+x_4^2$, then  by Lagrange's Four Square Theorem for every integer $n$ there exists $a=(a_1,a_2,a_3,a_4) \in \mathbb{Z}^4$ such that $Q(a)=n$, thus the congruence $Q(x) \equiv n \pmod{128}$ admits solutions for every $n \in \mathbb{N}$; this consideration allows us to discard this constraint in the following.

This lemma partially answers Question \ref{fourlarge}:
\begin{lemma}\label{point}
Let $n$ be a squarefree integer. Then, for $n$ sufficiently large, there is an integral point $(a_1,a_2,a_3,a_4) \in \mathcal{S}_n$ such that $a_i \ge \frac{\sqrt{n}}{8}$. 
\end{lemma}
\begin{proof}
Notice first that for $n$ squarefree we have $n \not \equiv 0 \pmod{4}$, hence $\sigma(n)=\sigma'(n)$ and $r(n) = 8\sigma(n) \ge 8(n+1)$.
 Considering the domain $$\mathcal{F}:=\left\{ (x_1,x_2,x_3,x_4) \in \mathcal{S}_n \ \bigg|  \ x_i \ge \frac{1}{8} \ \ \text{for every} \ i=1,2,3,4 \right\} \subseteq \mathcal{S}_1,$$ we have $|F| > 0$.  Thus $r_{\mathcal{F}}(n) \sim |\mathcal{F}| r(n) \ge |\mathcal{F}|8(n+1)$, therefore on squarefree integers the sequence $r_{\mathcal{F}}(n)$ is asymptotic to a sequence that is larger than $|\mathcal{F}|8n$, so the values of $r_{\mathcal{F}}(n)$ must be strictly positive for sufficiently large squarefree $n$, and this implies our thesis.
\end{proof}
A consequence of Lemma \ref{point} is:
\begin{lemma}\label{partial}
There exists a positive constant $K$ such that for every squarefree positive integer $n$ there is an integral point $(a_1,a_2,a_3,a_4) \in \mathcal{S}_n$ such that if $a_i \neq 0$ then $a_i \ge \frac{\sqrt{n}}{K}$.
\end{lemma}
\begin{proof}
By Lemma \ref{point} there exists $N \in \mathbb{N}$ such that for every squarefree positive integer $n \ge N$ there is an integral point $(a_1,a_2,a_3,a_4) \in \mathcal{S}_n$ with $a_i \ge \frac{\sqrt{n}}{8}$. Therefore there is at most a finite set $\{n_1,\ldots,n_t\}$ of squarefree positive integers for which that property should not hold. However, by Lagrange's Four Square Theorem there should be an integral point $(a_{i1},a_{i2},a_{i3},a_{i4})$ on the sphere
$\mathcal{S}_{n_i}$ for every $i=1,\ldots,t$. Define now for every $i=1,\ldots,t$ the constant $$K_i := \max \left\{\frac{\sqrt{n_i}}{a_{ij}} \ \bigg|  a_{ij} \neq 0 \ , \ j=1,\ldots,4\right\}.$$
Let now $K = \max \{K_1,\ldots,K_t,8\}+1$. Then for every squarefree positive integer $n$ we have that there is an integral point $(a_1,a_2,a_3,a_4) \in \mathcal{S}_n$ such that whenever $a_i \neq 0$ we have $a_i \ge \frac{\sqrt{n}}{\max\{K_1,\ldots,K_t,8\}} > \frac{\sqrt{n}}{K}$, thus concluding our proof.
\end{proof}
Now we are ready to prove the main theorem.
\begin{proof}[Proof of Theorem 1]
Let $n \in \mathbb{N}$, let $K$ be as in Lemma \ref{partial} and take $n' \ge K^2n^2$. Suppose that $n'= \rho^2 q$ where $\rho,q \in \mathbb{Z}^+$ and $q$ is squarefree. By Lemma \ref{partial} there is an integral point $(a_1,a_2,a_3,a_4)$ on the sphere $\mathcal{S}_q$ such that if $a_i \neq 0$  we must have $a_i \ge \frac{\sqrt{q}}{K}$. Since $a_1^2+a_2^2+a_3^2+a_4^2=q$ we have $(\rho a_1)^2 + (\rho a_2)^2 + (\rho a_3)^2 + (\rho a_4)^2 = n'$, and $\rho a_i \ge \frac{\rho \sqrt{q}}{C} = \frac{\sqrt{n'}}{C} \ge n$ whenever $\rho a_i \neq 0$. Thus $n' \in \Gamma_n$. This implies $F(\Gamma_n) < K^2 n^2$, hence $F(\Gamma_n)=O(n^2)$.
\end{proof}
\section{Asymptotics and computations}
Theorem \ref{DR} states the existence of a bound for $F(\Gamma_n)$: however, in the proof it is not explicited how large the \emph{sufficiently large} $n$ should be to satisfy Lemma \ref{point}, and thus we cannot define explicitly the constant $K$ of Lemma \ref{partial}, although the existence of such $K$ is sufficient for proving our main result. Here, we provide computational evidence that Question \ref{fourlarge} might be satisfied by considering $K=8$.
\begin{proposition}\label{even}
If there is an even positive integer $n$ such that $\tilde{K}$ is the minimal value of $K$ such that Question \ref{fourlarge} holds, then there is an infinite set of positive integers such that $\tilde{K}$ is the minimal value of $K$ such that Question \ref{fourlarge} holds. 
\end{proposition}
\begin{proof}
Consider the integer $4n$.
 
Clearly, if $a_1^2+a_2^2+a_3^2+a_4^2=n$ then $(2a_1)^2+(2a_2)^2+(2a_3)^2+(2a_4)^2=4n$ and $2a_i \ge \frac{\sqrt{4n}}{\tilde{K}}$, thus Question \ref{fourlarge} holds for $\tilde{K}$.
Suppose now that there exist $K' \le \tilde{K}$ and $b = (b_1,b_2,b_3,b_4) \in \mathcal{Z}_{4n}$ such that $b_i \ge \frac{\sqrt{4n}}{K'}=\frac{2\sqrt{n}}{K'}$ whenever $b_i \neq 0$. Since $b_1^2+b_2^2+b_3^2+b_4^2=n$ and $n$ is divisible by $8$, then every $b_i$ is even, and thus, considering $a_i=\frac{b_i}{2}$ for every $i=1,2,3,4$, we have $(a_1,a_2,a_3,a_4) \in \mathcal{Z}_n$ and $a_i \ge \frac{\sqrt{n}}{K'}$ whenever $a_i \neq 0$. However, $K' \le \tilde{K}$, thus leading to a contradiction. 

Therefore, $\tilde{K}$ is the minimum value of $K$ such that Question \ref{fourlarge} holds for $4n$. Since $4n$ is even, this implies that $\tilde{K}$ is the minimum value of $K$ such that Question \ref{fourlarge} holds for every integer of the form $4^\alpha n$, $\alpha \ge 0$, and thus our thesis is proved.
\end{proof} 

In light of this, we found the minimal integer constant $K$ that work for small values of $n$.
For this purpose, considering the set $\mathcal{Z}_n$ of integral points of $\mathcal{S}_n$, we define, for each point $a=(a_1,a_2,a_3,a_4) \in \mathcal{Z}_n$, $L_a=\min\{a_i \ | a_i \neq 0\}$. We can define a partial order $\le_L$ on $\mathcal{Z}_n$ such that $a \le_L b$ if and only if $L_a \le L_b$ for every $a,b \in \mathcal{Z}_n$. In order to get the minimal value of $K$ for a specific $n$, we need to find the maximal points of $(\mathcal{Z}_n,\le_L)$. 
We computed with GAP these maximal points for $n \le 2560000$. In Table \ref{tab1}, $I_K$ denotes the set of integers $n \le 2560000$ such that $K$ is the minimum value such that Question \ref{fourlarge} holds, and $S_K$ denotes the set of squarefree integers $n \le 2560000$ such that $K$ is the minimum value such that Question \ref{fourlarge} holds.
\begin{center}
\begin{table}[h]
    \begin{tabular}{| c || c | c | c |}
    \hline
    $K$ & $\# I_K$ & $\# S_K$ & $\max S_K$  \\ \hline \hline
    $1$ & $1600$ & $1$ & $1$  \\ \hline
    $2$ & $1996562$ & $1248201$ & $2559998$ \\ \hline
    $3$ & $561563$ & $308003$  & $2559999$ \\ \hline
    $4$ & $232$ & $59$ &  $1327$ \\ \hline
    $5$ & $21$ & $7$ & $151$ \\ \hline
    $6$ & $11$ & $3$ & $239$  \\ \hline
    $7$ & $9$ & $2$ & $46$ \\ \hline
    $8$ & $2$ & $2$ & $55$ \\ \hline 
    \end{tabular}
    \caption{}\label{tab1}
\end{table} 
\end{center}
From these computation it emerges that Question \ref{fourlarge} holds with $K=8$ for every $n \le 2560000$, and further considerations on this data suggest that $K=8$ might be a good choice for every integer $n$. Moreover, it seems that for the vast majority of integers $n$ computed, Question \ref{fourlarge} holds with $K \le 3$. However, since $46 \in I_7$, $30 \in I_6$, $78 \in I_5$ and $10 \in I_4$, Proposition \ref{even} assures that $I_4,I_5,I_6,I_7$ are all infinite sets; thus, surprisingly, even if we want to prove Question \ref{fourlarge} for large values of $n$ the constant $K$ required cannot be lower than $7$.

From Table \ref{tab1} it seems that for large squarefree integers the mininum constant $K$ required is either $2$ or $3$. Obviously $K$ cannot be lower than $3$, since if we consider squarefree integers $n$ such that $n \equiv 7 \pmod{8}$, thus for every $a=(a_1,a_2,a_3,a_4)$ such that $a_1^2+a_2^2+a_3^2+a_4^2=n$ we have $a_i \neq 0$ for every $i=1,2,3,4$ and thus the minimum $a_i$ cannot be greater than $\frac{\sqrt{n}}{2}$. Nevertheless, it might be interesting to see whether Question \ref{fourlarge} holds for squarefree integers with $K=3$.

Solving this Question would provide bounds for $F(\Gamma_n)$. Denoting with $F(n)$ the greatest positive integer that could not be represented as a sum of at most four squares of integers not lower than $n$, proving that Question \ref{fourlarge} holds for every integer with $K=8$ would imply that $F(\Gamma_n) \le F(n) \le 64n^2$. Because of the restriction on the number of squares involved in our sum, we cannot verify directly these bounds for $F(n)$. However, under the assumption that Question \ref{fourlarge} holds with $K=8$ for every positive integer, we can derive from the maximal vectors of positive integers up to $2560000=200^2 \cdot 64$ a table of the various $F(n)$ with $2 \le n \le 200$. 
Table \ref{tab2} compares these values of $F(n)$ with the relative values of $F(\Gamma_n)$.  
As Table \ref{tab2} shows, for $5 \le n \le 200$ it resulted that $F(n)=46 \cdot 4^{\lceil \log_2 n \rceil-1}$. Therefore, for these values it is easy to see that $F(n) \le 46n^2$, as Table \ref{tab2} and Figure \ref{fig} show. 
%
%
\begin{figure}[htp] 
\centering 
\hfill 
\begin{minipage}[b]{.3\columnwidth} 
  \centering 
      \begin{tabular}{| c || c | c |}
      \hline
      $n$ & $F(\Gamma_n)$ & $F(n)$   \\ \hline \hline
      $2$ & $23$ & $55$  \\ \hline
      $3$ & $87$ & $184$ \\ \hline
      $4$ & $119$ & $239$  \\ \hline
      $5$ & $201$ & $736$  \\ \hline
      $6$ & $312$ & $736$  \\ \hline
      $7$ & $376$ & $736$  \\ \hline
      $8$ & $455$ & $736$  \\ \hline
      $9$ & $616$ & $2944$  \\ \hline
      $10$ & $760$ & $2944$   \\ \hline
      $20$ & $2764$ & $11776$  \\ \hline
      $30$ & $5523$ & $11776$  \\ \hline
      $40$ & $9856$ & $47104$  \\ \hline
      $50$ & $15232$ & $47104$   \\ \hline
      $60$ & $21408$ & $47104$  \\ \hline
      $70$ & $28192$ & $188416$  \\ \hline
      $80$ & $36448$ & $188416$  \\ \hline
      $90$ & $46048$ & $188416$  \\ \hline
      $100$ & $57408$ & $188416$   \\ \hline
      $125$ & $84992$ & $188416$  \\ \hline
      $150$ & $122880$ & $753664$  \\ \hline
      $175$ & $164864$ & $753664$  \\ \hline
      $200$ & $215040$ & $753664$  \\ \hline

      \end{tabular}
      \captionof{table}{}\label{tab2}   
\end{minipage}\hfill
\begin{minipage}[b]{.5\columnwidth} 
  \centering 
  \includegraphics[scale=.6]{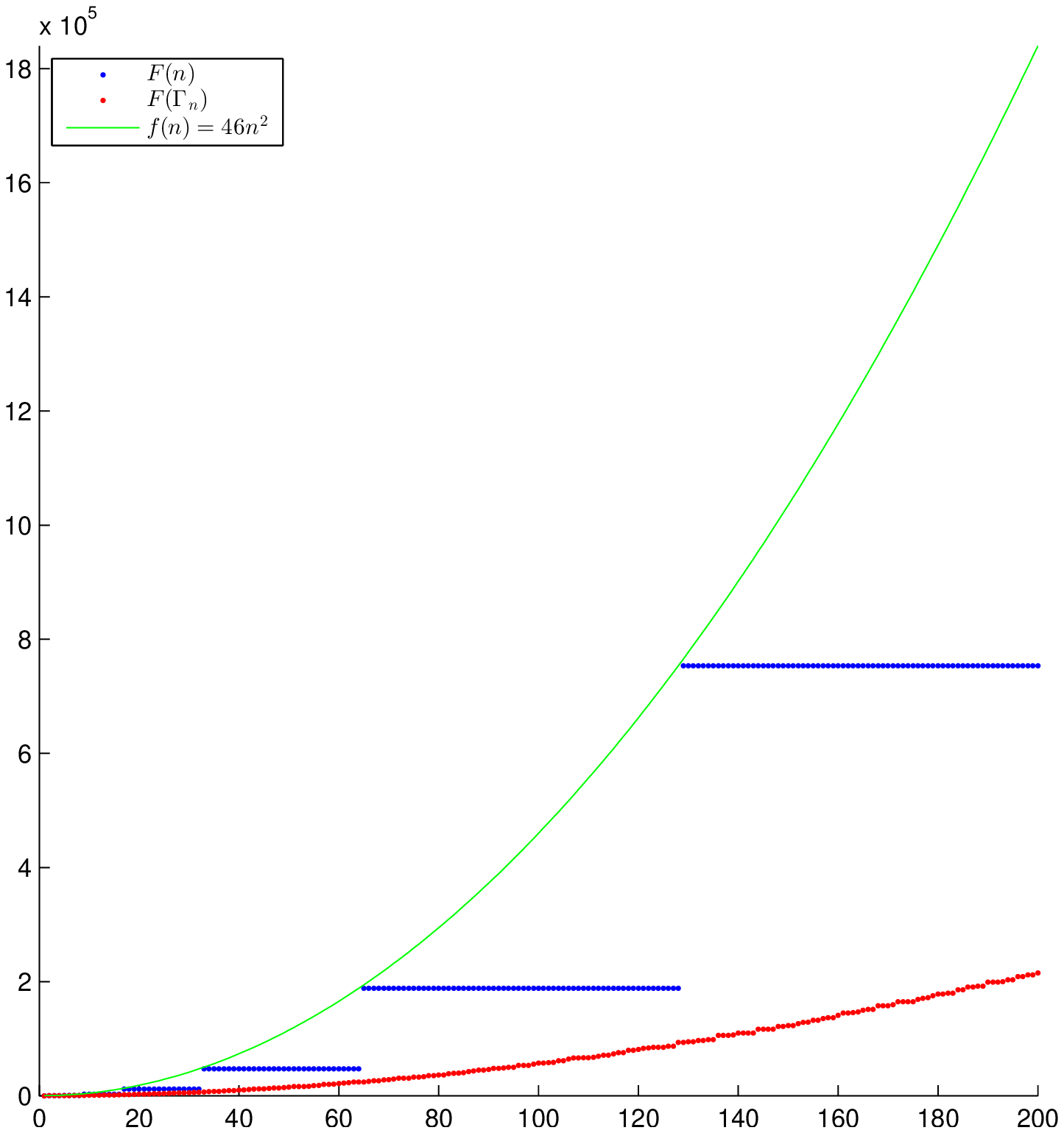} 
  \caption{}\label{fig} 
\end{minipage} 
\hspace*{\fill} 
\end{figure}
\section*{Acknowledgements}
I would like to thank Alessio Sammartano for introducing me to this problem, and for his helpful comments and suggestions on this work. I would also like to thank the referee for his useful remarks on Lemma \ref{point} and Section 2. 

\end{document}